\documentclass[10pt,a4paper]{article}

\usepackage[draft]{optional}

\usepackage{amsfonts, amsmath, wasysym}

\usepackage{tikz}

\usepackage{amssymb,amsthm,
paralist
}

\usepackage{
latexsym,
}

\usepackage{url}

\definecolor{darkgreen}{rgb}{0,0.5,0}
\definecolor{darkred}{rgb}{0.7,0,0}
\usepackage[colorlinks, 
citecolor=darkgreen, linkcolor=darkred
]{hyperref}

\theoremstyle{plain}

\numberwithin{equation}{section}

\newcommand{\la}{\lambda}

\renewcommand{\th}{\theta}

\newcommand{\vph}{\varphi}

\newcommand{\R}{\ensuremath{{\mathbb R}}}

\newcommand{\bx}{{\bf x}}

\newcommand{\beq}{\begin{equation}}
\newcommand{\eeq}{\end{equation}}
\newcommand{\beqa}{\begin{equation}\begin{aligned}}
\newcommand{\eeqa}{\end{aligned}\end{equation}}
\newcommand{\brmk}{\begin{rmk}}
\newcommand{\ermk}{\end{rmk}}
\newcommand{\partref}[1]{\hbox{(\csname @roman\endcsname{\ref{#1}})}}
\newcommand{\half}{\frac{1}{2}}

\theoremstyle{plain}
\newtheorem{thm}{Theorem} 

\theoremstyle{definition}
\newtheorem{rem}{Remark}

\title{\sc nontrivial breathers for ricci flow}
\author{Peter M. Topping}
\date{15 September 2021}

\begin{document}
\maketitle

\begin{abstract}
Perelman has proved that there cannot exist a nontrivial breather for Ricci flow on a closed manifold. Here we construct nontrivial expanding breathers 
for Ricci flow in all dimensions when the underlying manifold is allowed to be noncompact.
\end{abstract}


A complete Ricci flow $(M,g(t))$, for $t$ in some time interval $I$, is said to be a \emph{breather} if there exist distinct times $t_1,t_2\in I$, a scale $\la>0$ and a diffeomorphism $\vph:M\to M$ such that 
$$\vph^* g(t_1)=\la g(t_2).$$
A breather is said to be nontrivial if it is not also a Ricci soliton flow. In a Ricci soliton flow, the metric $g(t)$ is, up to scaling,  the pull-back of one fixed metric for \emph{every} $t\in I$.

In his celebrated work, Perelman \cite{P1} proved that there do not exist any nontrivial breathers on any closed manifold, completing an analysis started by Ivey \cite{ivey}.
Since then there have been many partial generalisations of this nonexistence result to the noncompact case, e.g. \cite{q_zhang, lu_zheng, y_zhang, cheng_zhang_TAMS, cheng_zhang}.
The purpose of this note is to give the first examples of nontrivial breathers.

\begin{thm}
\label{main_thm}
For each $\la>1$ 
there exists a smooth complete Ricci flow $g(t)$ on the punctured plane, for $t>0$, that is not a Ricci soliton, but which has the property that 
$$\vph^* g(t) = \la g(t/\la),$$
for all $t>0$, where $\vph$ is a diffeomorphism of the punctured plane to itself.
\end{thm}
The theorem induces examples in all dimensions $n\geq 2$. For $n\geq 3$ we can take the Cartesian product of our example with Euclidean space $\R^{n-2}$,
and extend the diffeomorphism $\vph$ to a diffeomorphism of $\R^n=\R^2\times \R^{n-2}$ to itself given by
$$(x,y)\mapsto (\vph(x), \sqrt{\la} y).$$

When proving this theorem we choose the diffeomorphism $\vph$ to be the dilation $\vph(\bx)=e\bx$ in order to keep the construction as clean as possible, but a slight variation of the argument allows us to change the scale factor $e$ to any positive value other than $1$.

The construction will rest on previous work establishing the existence and uniqueness of instantaneously complete Ricci flows starting with incomplete Riemannian metrics of unbounded curvature.
The following is an abridged version of the existence and uniqueness theory developed in 
\cite{GT2} and \cite{ICRF_UNIQ} respectively, following \cite{JEMS}.
\begin{thm}
\label{GT_thm}
Suppose $(M,g_0)$ is any smooth connected Riemannian surface with infinite volume (possibly incomplete, possibly noncompact and possibly with unbounded curvature).
Then there exists a unique smooth Ricci flow $g(t)$ on $M$, for $t\in [0,\infty)$, such that 
$g(0)=g_0$ and $g(t)$ is complete for all $t>0$.
\end{thm}
In order to understand the effect of this theorem it may help to consider the case that $(M,g_0)$ is the punctured plane with the Euclidean metric. In this special case the subsequent Ricci flow $g(t)$  provided by the theorem would be an expanding  Ricci soliton. A hyperbolic cusp of curvature $-1/(2t)$ instantaneously develops near to the puncture and gradually spreads out.

\begin{proof}[Proof of Theorem \ref{main_thm}]
The breather is to exist on the punctured plane, but this is conformally equivalent to the infinite cylinder $M=\R\times S^1$ via the diffeomorphism $(r,\th)\mapsto (\log r,\th)$ so we can work there. In coordinates $(x,\th)$, the dilation by a factor $e$ on the punctured plane becomes the diffeomorphism 
$$\vph(x,\th)=(x+1,\th),$$
that shifts along the cylinder $M$.
If we define a smooth metric 
$g_0=e^{2u(x,\th)}(dx^2+d\th^2)$ on $M$ by specifying
$$u(x,\th)=\half(\log\la)x+\sin 2\pi x$$
then 
$$\vph^* g_0 =e^{2u(x+1,\th)}(dx^2+d\th^2)= \la g_0.$$
Let $g(t)$, $t\in [0,\infty)$, be the unique instantaneously complete Ricci flow on $M$ starting with $g_0$ as given by Theorem \ref{GT_thm}. 
Then $\vph^* g(t)$ is also an instantaneously  complete Ricci flow, this time starting with 
$\vph^* g_0 =\la g_0$.
But the parabolic rescaling $\la g(t/\la)$ is also an instantaneously complete Ricci flow starting with 
the same initial metric $\la g_0$, and by uniqueness we must then have
$$\vph^* g(t) = \la g(t/\la)$$
as required.

There are multiple ways to verify that we have not inadvertently constructed a Ricci soliton. The flow we have constructed must be $\th$-invariant, by uniqueness, 
and any diffeomorphisms that induced a soliton structure would have to be conformal and (without loss of generality) could be taken to be translations of the cylinder. 
Therefore, if the flow were a soliton flow then it would have to be a \emph{gradient} soliton flow.
But such flows have been classified and do not coincide with our flow; see e.g. 
\cite{Bernstein_Mettler} and the references therein.
\end{proof}

\begin{rem}
The same ideas used here can be employed to find new examples of Ricci solitons, for example by flowing the flat half space for time $1$. However, the full power of this method can only be realised once one has extended the theory of instantaneously complete Ricci flows to cover more singular initial data. The resulting theory will be presented 
in \cite{TY3}.
\end{rem}

\begin{rem}
The same strategy as used here for Ricci flow can be used to construct nontrivial breathers for other scale-invariant geometric flows provided that a sufficiently strong existence and uniqueness theory is available. We consider this now for mean curvature flow, as was proposed to us by Felix Schulze.
As for Ricci flow, we only need work in the lowest dimension and consider curve shortening flow in the plane because we can always take a product with an additional Euclidean factor. 
The basic requirement is to write down a curve in the plane that is invariant under a dilation $\vph:\R^2\to\R^2$ such as $\vph(\bx)=e \bx$,  but not under every dilation, 
and that can be evolved uniquely under curve shortening flow.
The task is simplified a little by considering graphical solutions, and one suitable example of a curve would be the graph of the Lipschitz function $f:\R\to\R$ defined by
$$f(x)=\left\{
\begin{aligned}
|x|\sin(2\pi\log |x|) \quad &  \text{for } x\neq 0\\
0 \quad &  \text{for }x=0,
\end{aligned}
\right. $$
as illustrated in Figure \ref{fig:f}.
This curve has an evolution under curve shortening flow that exists for all time
(see for example \cite{EH1} or \cite{CZ}) and this solution can be written as the graph of a time dependent function $F(\cdot,t)$ where $F:\R\times [0,\infty)\to\R$ is 
smooth away from $(0,0)$ where it is nevertheless continuous. 
This evolution is seen to be unique either by a simple barrier and maximum principle argument or via a more general result such as \cite[Proposition 5.1]{CZ}.
Because $f(x)=\frac{1}{e}f(ex)$, we find that the parabolically scaled solution 
$\frac{1}{e}F(ex,e^2t)$ has the same initial data as the flow $F(x,t)$, and by uniqueness we deduce that
$$F(x,t)=\frac{1}{e}F(ex,e^2t).$$
This says that if we pull back the evolution in the plane by the dilation $\vph$ then we get the same flow as before with time scaled by a factor $e^2$, which makes it a breather. 
\end{rem}

\noindent
\emph{Acknowledgements:}
This work was supported by EPSRC grant reference EP/T019824/1.

\begin{figure}
\centering
\begin{tikzpicture}

\draw [<->] (-5,0) -- (5,0) node[below]{$x$} ;
\draw [<->] (0,-4) -- (0,4) node[left]{$f(x)$} ;

\draw[thick, domain=0.001:0.5, samples=500] 
plot (\x, {
(abs(\x)*sin(deg(2*3.142*ln(\x))))
});
\draw[thick, domain=0.5:5, samples=200] 
plot (\x, {
(abs(\x)*sin(deg(2*3.142*ln(\x))))
});
\draw[thick, domain=-0.001:-0.5, samples=500] 
plot (\x, {
(-\x*sin(deg(2*3.142*ln(-\x))))
});
\draw[thick, domain=-0.5:-5, samples=200] 
plot (\x, {
(-\x*sin(deg(2*3.142*ln(-\x))))
});

\end{tikzpicture}
\caption{Curve shortening flow breather}
\label{fig:f}
\end{figure}
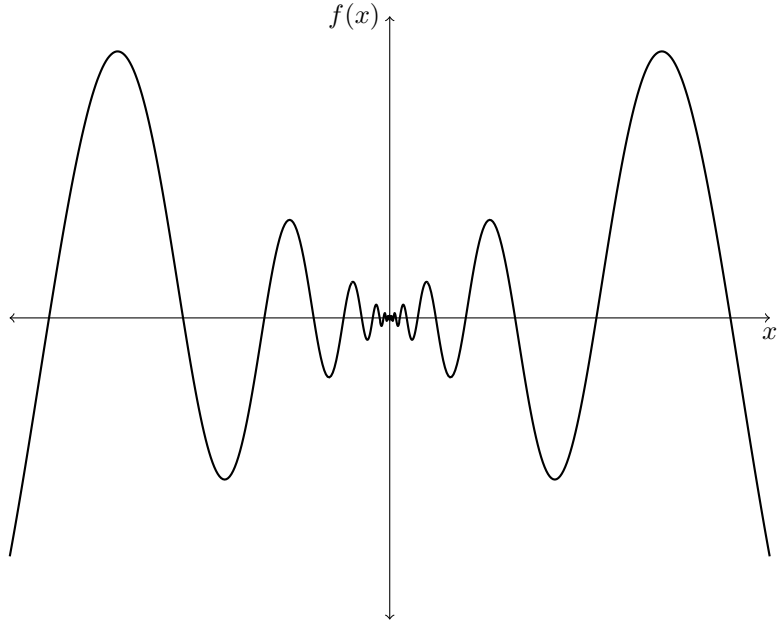

\noindent
\url{https://homepages.warwick.ac.uk/~maseq/}

\noindent
{\sc Mathematics Institute, University of Warwick, Coventry,
CV4 7AL, UK.}

\end{document}